\documentclass[a4paper]{article}

\usepackage{amsthm}
\usepackage{amsmath}
\usepackage{bm}
\usepackage{booktabs}
\usepackage{enumitem}
\usepackage[hang,flushmargin]{footmisc}
\usepackage[left=25mm,right=25mm,top=25mm,bottom=25mm]{geometry}
\usepackage[
	pdftitle={Monotone orbifold Hurwitz numbers},
	pdfauthor={Norman Do and Maksim Karev},
	ocgcolorlinks,
	linkcolor=linkblue,
	citecolor=linkred,
	urlcolor=linkred]
{hyperref}
\usepackage{mathpazo}
\usepackage{mathtools}
\usepackage{microtype}
\usepackage{multicol}
\usepackage{sectsty}
\usepackage{setspace}
\usepackage{tikz}
	\usetikzlibrary{positioning}
\usepackage{titlesec}
	\titlespacing{\section}{0pt}{12pt}{0pt}
	\titlespacing{\subsection}{0pt}{6pt}{0pt}
\usepackage[nottoc]{tocbibind}

\theoremstyle{plain}
\newtheorem{theorem}{Theorem}
\newtheorem{proposition}[theorem]{Proposition}

\newtheorem{lemma}[theorem]{Lemma}
\newtheorem{conjecture}[theorem]{Conjecture}
\newtheorem*{conjecture*}{Conjecture}

\theoremstyle{definition}
\newtheorem{definition}[theorem]{Definition}

\theoremstyle{remark}
\newtheorem{remark}[theorem]{Remark}

\definecolor{linkred}{rgb}{0.75,0,0}
\definecolor{linkblue}{rgb}{0,0,1}

\newcommand\blfootnote[1]{
	\begingroup
	\renewcommand\thefootnote{}\footnote{#1}
	\addtocounter{footnote}{-1}
	\endgroup
}

\sectionfont{\large}
\subsectionfont{\normalsize}

\setlist{nolistsep}

\newcommand{\dd}{\mathrm{d}}
\newcommand{\h}{\hbar}
\newcommand{\xx}{\widehat{x}}
\newcommand{\yy}{\widehat{y}}
\newcommand{\z}{\overline{z}}
\newcommand{\zz}{\mathbf{z}}
\renewcommand{\H}{\vec{H}}

\newcommand{\mmu}{\boldsymbol{\mu}}
\newcommand{\vbar}{\,\vert\,}
\DeclareRobustCommand{\stirling}{\genfrac\{\}{0pt}{}}

\makeatletter
\newcommand{\dashrule}[1][black]{
	\color{#1}\rule[\dimexpr.5ex-.2pt]{4pt}{.4pt}\xleaders\hbox{\rule{4pt}{0pt}\rule[\dimexpr.5ex-.2pt]{4pt}{.4pt}}\hfill\kern0pt
}
\newcommand{\rulecolor}[1]{
	\def\CT@arc@{\color{#1}}
}
\makeatother

\begin{document}

{\large \bfseries Monotone orbifold Hurwitz numbers}

{\bfseries Norman Do and Maksim Karev}

\emph{Abstract.} In general, Hurwitz numbers count branched covers of the Riemann sphere with prescribed ramification data, or equivalently, factorisations in the symmetric group with prescribed cycle structure data. In this paper, we initiate the study of monotone orbifold Hurwitz numbers. These are simultaneously variations of the orbifold case and generalisations of the monotone case, both of which have been previously studied in the literature. We derive a cut-and-join recursion for monotone orbifold Hurwitz numbers, determine a quantum curve governing their wave function, and state an explicit conjecture relating them to topological recursion.

\blfootnote{
\emph{Date:} 20 May 2015 \\ The first author was partly supported by the Australian Research Council grant DE130100650. The second author was partly supported by the Russian Foundation for Basic Research grant 13-01-00383a.}

\hrule

\setlength{\parskip}{0pt}
\tableofcontents
\setlength{\parskip}{6pt}

\section{Introduction} \label{sec:intro}

In general, Hurwitz numbers count branched covers of the Riemann sphere with prescribed genus and ramification data. For example, the simple Hurwitz number $H_{g,n}(\mu_1, \ldots, \mu_n)$ is the weighted count of genus $g$ connected branched covers $(\Sigma; p_1, \ldots, p_n) \to (\mathbb{CP}^1; \infty)$ with $m$ fixed points with simple ramification and branching order $\mu_i$ at the preimage $p_i$ of $\infty$. As usual, the weight of a branched cover is the reciprocal of its number of automorphisms.

The Riemann existence theorem asserts that there is a unique branched cover with prescribed monodromy at the ramification points. It follows that the simple Hurwitz number $H_{g,n}(\mu_1, \ldots, \mu_n)$ is equal to $\frac{1}{|\mmu|!}$ multiplied by the number of tuples $(\sigma_1, \ldots, \sigma_m)$ of transpositions in $S_{|\mmu|}$ such that
\begin{itemize}
\item $m = 2g - 2 + n + |\mmu|$;
\item the cycles of $\sigma_1 \cdots \sigma_m$ are labelled $1, 2, \ldots, n$ such that cycle $i$ has length $\mu_i$; and
\item $\sigma_1, \ldots, \sigma_m$ generate a transitive subgroup of $S_{|\mmu|}$.
\end{itemize}
Here and throughout the paper, we use the notation $|\mmu|$ to denote the sum $\mu_1 + \cdots + \mu_n$. The first condition above is implied by the Riemann--Hurwitz theorem, the second is a consequence of the ramification profile over $\infty$, and the third guarantees that the resulting branched cover is connected.

Since Hurwitz first introduced them~\cite{hur}, simple Hurwitz numbers have been well-studied in the literature. In particular, the last two decades have seen a remarkable revival of Hurwitz theory, inspired by connections to enumerative geometry and mathematical physics. For example, it was observed that the Hurwitz numbers exhibit polynomiality~\cite{gou-jac}, which was later derived as a consequence of the celebrated ELSV theorem. This result not only proves that simple Hurwitz numbers are inherently polynomial, but that the coefficients of the polynomials are intersection numbers on the Deligne--Mumford moduli spaces of curves~\cite{eke-lan-sha-vai}. More recently, the simple Hurwitz numbers have been shown to obey the topological recursion of Chekhov, Eynard and Orantin~\cite{eyn-mul-saf}, as well as a quantum curve equation~\cite{mul-sha-spi}.

There are several variations of the simple Hurwitz numbers. For example, one can require the ramification profile over $0 \in \mathbb{CP}^1$ to be of type $(a, a, \ldots, a)$ for a fixed positive integer $a$. In the monodromy viewpoint, one must include another permutation $\sigma_0$ with cycle type $(a, a, \ldots, a)$ in the factorisation. The resulting enumeration produces so-called \emph{$a$-orbifold Hurwitz numbers} and one recovers the usual simple Hurwitz numbers in the case $a = 1$. Various results concerning simple Hurwitz numbers have analogues in the case of orbifold Hurwitz numbers, such as polynomiality~\cite{joh-pan-tse, dun-lew-pop-sha}, the topological recursion~\cite{do-lei-nor, bou-her-liu-mul}, and a quantum curve equation~\cite{mul-sha-spi}.

Another variation arises by requiring the sequence of transpositions $\sigma_1, \ldots, \sigma_m$ to be \emph{monotone}. In other words, if we write $\sigma_i = (r_i ~ s_i)$ with $r_i < s_i$ for $i = 1, 2, \ldots, m$, then we have $s_1 \leq s_2 \leq \cdots \leq s_m$. The resulting enumeration produces so-called \emph{monotone Hurwitz numbers}, which were introduced by Goulden, Guay-Paquet and Novak in their work on the asymptotic expansion of the HCIZ matrix integral~\cite{gou-gua-nov14}. Again, various results concerning simple Hurwitz numbers have analogues in the case of monotone Hurwitz numbers, such as polynomiality~\cite{gou-gua-nov13b}, the topological recursion~\cite{do-dye-mat}, and a quantum curve equation~\cite{do-dye-mat}.

In this paper, we introduce the notion of a monotone orbifold Hurwitz number, which is a hybrid of the aforementioned variations and defined as follows.

\begin{definition}
The {\em monotone orbifold Hurwitz number $\H_{g,n}^{[a]}(\mu_1, \ldots, \mu_n)$} is equal to $\frac{1}{|\mmu|!}$ multiplied by the number of tuples $(\sigma_0, \sigma_1, \ldots, \sigma_m)$ of permutations in $S_{|\mmu|}$ such that
\begin{itemize}
\item $m = 2g - 2 + n + \frac{|\mmu|}{a}$;
\item $\sigma_0$ has cycle type $(a, a, \ldots, a)$ and $\sigma_1, \ldots, \sigma_m$ is a monotone sequence of transpositions;
\item the cycles of $\sigma_0 \sigma_1 \cdots \sigma_m$ are labelled $1, 2, \ldots, n$ such that cycle $i$ has length $\mu_i$; and
\item $\sigma_0, \sigma_1, \ldots, \sigma_m$ generate a transitive subgroup of $S_{|\mmu|}$.
\end{itemize}
\end{definition}

It is natural to ask which results for simple Hurwitz numbers have analogues in the case of monotone orbifold Hurwitz numbers. We begin to answer this question by deriving a cut-and-join recursion for monotone orbifold Hurwitz numbers. The cut-and-join analysis that we use to derive the recursion requires a refined enumeration $\H_{g,n}^{[a],\ell}(\mu_1 \vbar \mu_2, \ldots, \mu_n)$ for $\ell = 1, 2, \ldots, a$, which satisfies
\begin{equation} \label{eq:hurwitzsum1}
\H_{g,n}^{[a]}(\mu_1, \ldots, \mu_n) = \sum_{i=1}^n \sum_{\ell=1}^a \H_{g,n}^{[a],\ell}(\mu_i \vbar \mmu_{S \setminus \{i\}}).
\end{equation}
We use the vertical bar in the notation $\H_{g,n}^{[a],\ell}(\mu_1 \vbar \mu_2, \ldots, \mu_n)$ to emphasise that $\mu_1$ is to be treated as special among the arguments $\mu_1, \mu_2, \ldots, \mu_n$. The definition for this refined enumeration and the proof of the following theorem will be presented in detail in Section~\ref{sec:cutjoin}.

\begin{theorem}[Cut-and-join recursion] \label{thm:cutjoin}
The monotone orbifold Hurwitz numbers are completely determined by equation~\eqref{eq:hurwitzsum1}, the base case $\H_{0,1}^{[a],\ell}(a \vbar) = \frac{1}{a} \delta_{\ell, 1}$, and the cut-and-join recursion
\begin{align} \label{eq:cutjoin}
\H_{g,n}^{[a],\ell}(\mu_1\vbar \mmu_{S\setminus \{1\}}) =&\, \Theta(\mu_1 + \ell - a -1) \sum_{p=1}^\ell \sum_{i=2}^n\H_{g,n-1}^{[a],p}(\mu_1+\mu_i \vbar \mmu_{S\setminus \{1,i\}})\\
&+\sum_{\alpha + \beta = \mu_1} \sum_{p=1}^\ell \beta \, \H_{g-1,n+1}^{[a],p}(\alpha \vbar \mmu_{S\setminus \{1\}}, \beta) \nonumber \\ 
&+ \sum_{\alpha+ \beta = \mu_1}\mathop{\sum_{g_1 + g_2 = g}}_{I \sqcup J = S \setminus \{1\}}\sum_{p=1}^\ell \frac{|\mmu_J| + \alpha}{|\mmu|} \beta \, \H_{g_1,|I| + 1}^{[a]}(\mmu_I, \beta) \, \H_{g_2,|J|+1}^{[a],p}(\alpha \vbar \mmu_{ J}). \nonumber
\end{align}
Here, $\Theta$ denotes the Heaviside step function.
\end{theorem}

Analogous to other problems in Hurwitz theory and enumerative geometry, we define the \emph{free energies} to be the following generating functions, for $g \geq 0$ and $n \geq 1$.
\[
\vec{F}_{g,n}^{[a]}(x_1, \ldots, x_n) = \sum_{\mu_1, \ldots, \mu_n = 1}^\infty \H_{g,n}^{[a]}(\mu_1, \ldots, \mu_n) \prod_{i=1}^n x_i^{\mu_i}
\]
It is common to assemble the free energies into the following \emph{wave function}.
\[
\vec{Z}^{[a]}(x, \h) = \exp \left[ \sum_{g=0}^\infty \sum_{n=1}^\infty \frac{\h^{2g-2+n}}{n!} \, \vec{F}_{g,n}^{[a]}(x, \ldots, x) \right]
\]
One can interpret the wave function as a generating function for possibly disconnected branched covers, assembled according to the degree and the Euler characteristic $2g-2+n$. It follows that its coefficients simply enumerate sequences of monotone transpositions. It is this observation that allows us to deduce the following result.

\begin{theorem}[Quantum curve] \label{thm:qcurve}
The wave function for monotone orbifold Hurwitz numbers is given by
\[
\vec{Z}^{[a]}(x, \h) = 1 + \sum_{k=1}^\infty \frac{x^{ak}}{k! \, a^k \, \h^k} \prod_{j=1}^{ak-1} \frac{1}{1-j\h}.
\]
Furthermore, it satisfies the following differential equation, where $\xx = x$ and $\yy = -\h \frac{\partial}{\partial x}$ are multiplication and differentiation operators.
\[
\left[ \xx^{a-1} + \prod_{j=0}^{a-1} (1 + \xx \yy + j\h) \, \yy \right] \vec{Z}^{[a]}(x, \h) = 0
\]
\end{theorem}

The semi-classical limit of the quantum curve is obtained by setting $\h = 0$ and replacing the operators $\xx$ and $\yy$ with commuting variables $x$ and $y$, respectively. From this procedure, one obtains the algebraic curve
\[
x^{a-1} + y(1+xy)^a = 0,
\]
which has the rational parametrisation
\begin{equation} \label{eq:scurve}
x(z) = z(1-z^a) \qquad \qquad \text{and} \qquad \qquad y(z) = \frac{z^{a-1}}{z^a-1}.
\end{equation}

The topological recursion of Chekhov, Eynard and Orantin~\cite{che-eyn,eyn-ora07} takes as input the data of a spectral curve and outputs multidifferentials $\omega_{g,n}$ for $g \geq 0$ and $n \geq 1$. Physics-inspired arguments suggest that applying the topological recursion to the spectral curve obtained as the semi-classical limit of a quantum curve should reproduce the corresponding free energies~\cite{guk-sul}. This observation leads directly to the following conjecture.

\begin{conjecture}[Topological recursion] \label{con:toprec}
Topological recursion applied to the spectral curve defined by equation~\eqref{eq:scurve} produces correlation differentials that satisfy
\[
\omega_{g,n} = \sum_{\mu_1, \ldots, \mu_n = 1}^\infty \H_{g,n}^{[a]}(\mu_1, \ldots, \mu_n) \prod_{i=1}^n \mu_i x_i^{\mu_i-1} \, \dd x_i, \qquad \text{for } (g,n) \neq (0,2).
\]
\end{conjecture}

The structure of the paper is as follows.
\begin{itemize}
\item In Section~\ref{sec:hurwitz}, we briefly review some of the known results on simple Hurwitz numbers and their variations. These form the inspiration for this paper, which discusses preliminary work towards finding analogues of these results in the case of monotone orbifold Hurwitz numbers.
\item In Section~\ref{sec:cutjoin}, we derive the cut-and-join recursion of Theorem~\ref{thm:cutjoin}, which allows one to recursively calculate monotone orbifold Hurwitz numbers. The analysis has been carried out using the language of monotone monodromy graphs, which bear similarity to combinatorial structures arising in tropical geometry.
\item In Section~\ref{sec:qcurve}, we prove Theorem~\ref{thm:qcurve}, which determines the quantum curve for monotone orbifold Hurwitz numbers. 
\item In Section~\ref{sec:toprec}, we discuss Conjecture~\ref{con:toprec}, which relates the monotone orbifold Hurwitz numbers to the topological recursion applied to an explicit spectral curve. Furthermore, we provide some evidence to support the conjecture.
\end{itemize}

The authors would like to thank the organisers of the conference ``Embedded Graphs'' (St. Petersburg, October 2014) during which this project was initiated, as well as Ga\"etan Borot, Hannah Markwig, Fedor Petrov, and Johannes Rau for fruitful discussions. The authors would also like to thank the anonymous referee whose remarks helped to make the text more readable.

\section{Hurwitz numbers: simple, orbifold, and monotone} \label{sec:hurwitz}

The simple Hurwitz number $H_{g,n}(\mu_1, \ldots, \mu_n)$ is the weighted count of genus $g$ connected branched covers $(\Sigma; p_1, \ldots, p_n) \to (\mathbb{CP}^1; \infty)$ with $m$ fixed points with simple ramification and branching order $\mu_i$ at the preimage $p_i$ of $\infty$. One may attach a monodromy permutation to each ramification point and invoke the Riemann existence theorem to show that the simple Hurwitz number equivalently counts certain factorisations in the symmetric group $S_{|\mmu|}$. More precisely, we make the following definitions.

\begin{definition} \label{def:fact}
Let $\mmu = (\mu_1, \ldots, \mu_n)$ be a tuple of positive integers. A \emph{Hurwitz factorisation of type $(g, \mmu)$} is a tuple $(\sigma_0, \sigma_1, \ldots, \sigma_m)$ of permutations in $S_{|\mmu|}$ such that
\begin{itemize}
\item $m = 2g-2+n+c(\sigma_0)$, where $c(\sigma_0)$ denotes the number of cycles in the permutation $\sigma_0$;
\item $\sigma_1, \ldots, \sigma_m$ are transpositions;
\item the cycles of $\sigma_0 \sigma_1 \cdots \sigma_m$ are labelled $1, 2, \ldots, n$ such that cycle $i$ has length $\mu_i$; and
\item $\sigma_0, \sigma_1, \ldots, \sigma_m$ generate a transitive subgroup of $S_{|\mmu|}$.
\end{itemize}
If $\sigma_0$ is equal to the identity, then we call the factorisation \emph{simple}. If $\sigma_0$ has cycle type $(a, a, \ldots, a)$, then we call the factorisation \emph{$a$-orbifold}. If $\sigma_1, \ldots, \sigma_m$ is a monotone sequence of transpositions, then we refer to the factorisation as \emph{monotone}. Recall that $\sigma_1, \ldots, \sigma_m$ is a monotone sequence of transpositions if $\sigma_i = (r_i ~ s_i)$ with $r_i < s_i$ for $i = 1, 2, \ldots, m$ and $s_1 \leq s_2 \leq \cdots \leq s_m$.
\end{definition}

\begin{definition}
The \emph{simple Hurwitz number} $H_{g,n}(\mmu)$ is $\frac{1}{|\mmu|!}$ multiplied by the number of simple Hurwitz factorisations of type $(g, \mmu)$. We similarly define the \emph{$a$-orbifold Hurwitz number} $H_{g,n}^{[a]}(\mmu)$, the \emph{monotone Hurwitz number} $\H_{g,n}(\mmu)$, and the \emph{monotone $a$-orbifold Hurwitz number} $\H_{g,n}^{[a]}(\mmu)$.
\end{definition}

When calculating monotone orbifold Hurwitz numbers, it is convenient to use the following observation. There is an equal number of monotone orbifold Hurwitz factorisations $(\sigma_0, \sigma_1, \ldots, \sigma_m)$ of type $(g, \mmu)$ for each choice of $\sigma_0$ of cycle type $(a, a, \ldots, a)$. This is a consequence of the following result.

\begin{lemma} \label{lem:cycletype}
The number of factorisations $\sigma \sigma_1 \cdots \sigma_m = \tau$, where $\sigma$ is a fixed permutation, $\tau$ is a permutation of cycle type $\mmu$, and $\sigma_1, \ldots, \sigma_m$ is a monotone sequence of transpositions, depends only on the cycle type of $\sigma$. Moreover, the result still holds if we restrict to transitive monotone factorisations --- in other words, those in which $\sigma, \sigma_1, \ldots, \sigma_m$ generate a transitive subgroup of $S_{|\mmu|}$.
\end{lemma}

\begin{proof}
Let $K_{m, \mmu}^\bullet(\sigma)$ be the number of monotone factorisations $\sigma \sigma_1 \cdots \sigma_m = \tau$, where $\tau$ has cycle type $\mmu$. Rewrite the factorisation as $\sigma_1 \cdots \sigma_m \tau^{-1} = \sigma^{-1}$ and observe that $K_{m, \mmu}^\bullet(\sigma)$ is simply the coefficient of $\sigma^{-1}$ in the element of the symmetric group algebra
\begin{equation} \label{eq:jmelements}
h_m(J_2, J_3, \ldots, J_{|\mmu|}) \, C_{\mmu} \in \mathbb{C}[S_{|\mmu|}].
\end{equation}
Here, $C_{\mmu}$ denotes the conjugacy class of permutations with cycle type $\mmu$, $h_m$ is the complete homogeneous symmetric function of degree $m$, and $J_2, J_3, \ldots$ denote the Jucys--Murphy elements
\[
J_k = (1~k) + (2~k) + \cdots + (k-1 ~ k) \in \mathbb{C}[S_{|\mmu|}], \qquad \text{for } k = 2, 3, \ldots, |\mmu|.
\]
Now we simply use the well-known fact that conjugacy classes and symmetric functions of the Jucys--Murphy elements live in the centre $Z\mathbb{C}[S_{|\mmu|}]$. So the expression in equation~\eqref{eq:jmelements} is a linear combination of conjugacy classes and it follows that each permutation in a given conjugacy class appears with the same coefficient. Therefore, $K_{m, \mmu}^\bullet(\sigma)$ depends only on the cycle type of $\sigma$.

Now let $K_{m, \mmu}^\circ(\sigma)$ be the analogous enumeration restricted to transitive monotone factorisations. If $\sigma$ is a cycle, then $K_{m, \mmu}^\circ(\sigma) = K_{m, \mmu}^\bullet(\sigma)$ and the result holds. So suppose now that $\sigma$ is the disjoint union of $k \geq 2$ cycles, which we write as $\sigma = C_1 C_2 \cdots C_k$.

Observe that every monotone factorisation can be equivalently interpreted as a union of  transitive monotone factorisations, by considering the maximal subsets on which $\langle \sigma, \sigma_1, \ldots, \sigma_m \rangle$ acts transitively. This leads to the decomposition
\[
K_{m, \mmu}^\bullet(\sigma) = K_{m, \mmu}^\circ(\sigma)  + \sum_{s=2}^k \mathop{\mathop{\sum_{I_1 \sqcup \cdots \sqcup I_s = [k]}}_{\mmu^{(1)} \sqcup \cdots \sqcup \mmu^{(s)} = \mmu}}_{m_1 + \cdots + m_s = m} \prod_{i=1}^s K_{m_i, \mmu^{(i)}}^\circ(C_{I_i}).
\]
The inner summation is over unordered partitions $I_1 \sqcup \cdots \sqcup I_s$ of $[k] = \{1, 2, \ldots, k\}$ into non-empty subsets, ordered tuples of  partitions $\mmu^{(1)} \sqcup \cdots \sqcup \mmu^{(s)}$ whose union is $\mmu$, and compositions $m_1 + \cdots + m_s$ of $m$ into positive integers. For $I \subseteq [k]$, we let $C_{I}$ denote the permutation that is the product of the disjoint cycles $C_i$ for $i \in I$ and for $J \subseteq [\ell]$, we let $\mmu$ denote the partition whose parts are $\mu_j$ for $j \in J$.

We have already deduced that $K_{m, \mmu}^\bullet(\sigma)$ depends only on the cycle type of $\sigma$. Furthermore, by induction on the number of cycles of $\sigma$, we know that the terms $K_{m_i, \mmu^{(i)}}^\circ(C_{I_i})$ appearing on the right side of the equation depend only on the cycle type of the permutation $C_{I_i}$. It follows that the remaining term $K_{m, \mmu}^\circ(\sigma)$ also depends only on the cycle type of $\sigma$.
\end{proof}

An alternative proof of Lemma~\ref{lem:cycletype} arises via the construction of a bijection between the monotone factorisations $(\sigma, \sigma_1, \cdots, \sigma_m)$ and the monotone factorisations $(\sigma', \sigma'_1, \cdots \sigma'_m)$, in which $\sigma$ and $\sigma'$ have the same cycle type while $\sigma\sigma_1\cdots \sigma_m$ and $\sigma'\sigma_1'\cdots \sigma'_m$  have cycle type $\mmu$. The basic idea is to conjugate each term of the first factorisation by the same permutation $\rho$ such that $\rho \sigma \rho^{-1} = \sigma'$, to obtain a factorisation that is not necessarily monotone. One then uses the natural action of the braid group $B_m$ on the set of factorisations with $m$ transpositions. Here, the $i$th braid group generator acts on a factorisation by sending the pair of transpositions $(\sigma_i, \sigma_{i+1})$ to $(\sigma_i \sigma_{i+1} \sigma_i, \sigma_i)$. It is possible to associate to the pair $( \rho,(\sigma, \sigma_1,\ldots, \sigma_m))$ an element of $B_m$ that sends the factorisation $(\sigma, \sigma_1, \ldots, \sigma_m)$ termwise conjugated by $\rho$  to a monotone factorisation in a canonical way, thus determining the required bijection. Furthermore, the braid group action preserves transitivity of factorisations. An explicit construction of the corresponding element of $B_m$ is beyond the scope of this paper, and we leave it as an exercise for the interested reader.

In the remainder of this section, we enumerate a number of interesting properties enjoyed by the simple Hurwitz numbers. Furthermore, we state generalisations to orbifold Hurwitz numbers and analogues for monotone Hurwitz numbers, where such results are known.

\subsection*{Cut-and-join recursion}

The cut-and-join recursion of Goulden and Jackson expresses a simple Hurwitz number in terms of simple Hurwitz numbers enumerating branched covers with fewer ramification points~\cite{gou-jac}. The combinatorial mechanism for the recursion comes from the elementary observation that
\[
\sigma_1 \sigma_2 \cdots \sigma_m = \tau \qquad \Rightarrow \qquad \sigma_1 \sigma_2 \cdots \sigma_{m-1} = \tau \sigma_m.
\]
Composing with the transposition $\sigma_m = (r~s)$ cuts one of the cycles of $\tau$ into two when $r$ and $s$ belong to the same cycle of $\tau$ and joins two of the cycles of $\tau$ into one when $r$ and $s$ belong to different cycles of $\tau$. At the level of branched covers, one can interpret this process as sending one of the simple ramification points to $\infty \in \mathbb{CP}^1$. This analysis also applies to orbifold Hurwitz numbers and the result is the following.

\begin{theorem}[Cut-and-join recursion for orbifold Hurwitz numbers~\cite{gou-jac, do-lei-nor}]
Fix a positive integer $a$ and consider the normalisation $K_{g,n}(\mu_1, \ldots, \mu_n) = \frac{1}{m!} H_{g,n}^{[a]}(\mu_1, \ldots, \mu_n)$, where $m = 2g-2+n+\frac{|\mmu|}{a}$. These numbers satisfy the recursion
\begin{align*}
m K_{g,n}(\mu_1, \ldots, \mu_n) &= \sum_{i < j} (\mu_i + \mu_j) \, K_{g,n-1}(\mmu_{S \setminus \{i,j\}}, \mu_i+\mu_j) \\
&+ \frac{1}{2} \sum_{i=1}^n \sum_{\alpha + \beta = \mu_i} \alpha \beta \bigg[ K_{g-1,n+1}(\mmu_{S \setminus \{i\}}, \alpha, \beta) + \mathop{\sum_{g_1+g_2=g}}_{I \sqcup J = S \setminus \{i\}} K_{g_1, |I|+1}(\mmu_I, \alpha) \, K_{g_2, |J|+1}(\mmu_J, \beta) \bigg],
\end{align*}
where $S = \{1, 2, \ldots, n\}$ and $\mmu_I = \{\mu_{i_1}, \ldots, \mu_{i_k}\}$ for $I = \{i_1, \ldots, i_k\}$.
\end{theorem}

A similar cut-and-join analysis applies to the monotone Hurwitz numbers, in which case one obtains the following result.

\begin{theorem}[Cut-and-join recursion for monotone Hurwitz numbers~\cite{gou-gua-nov13a}]
The monotone Hurwitz numbers satisfy the recursion
\begin{align*}
\mu_1 \H_{g,n}(\mu_1, \ldots, \mu_n) &= \sum_{i=2}^n (\mu_1 + \mu_i) \, \H_{g,n-1}(\mmu_{S \setminus \{1, i\}}, \mu_1+\mu_i) \\
&+ \sum_{\alpha + \beta = \mu_1} \alpha \beta \bigg[ \H_{g-1,n+1}(\mmu_{S \setminus \{i\}}, \alpha, \beta) + \mathop{\sum_{g_1+g_2=g}}_{I \sqcup J = S \setminus \{1\}} \H_{g_1, |I|+1}(\mmu_I, \alpha) \, \H_{g_2, |J|+1}(\mmu_J, \beta) \bigg].
\end{align*}
\end{theorem}

\subsection*{Polynomiality and the ELSV formula}

It was observed by Goulden, Jackson and Vainshtein~\cite{gou-jac-vai} that for $(g,n) \neq (0,1)$ or $(0,2)$, there exists a symmetric polynomial $P_{g,n}$ of degree $3g-3+n$ such that
\[
H_{g,n}(\mu_1, \ldots, \mu_n) = m! \prod_{i=1}^n \frac{\mu_i^{\mu_i}}{\mu_i!} P_{g,n}(\mu_1, \ldots, \mu_n).
\]
Although inherently a combinatorial statement, the polynomiality of simple Hurwitz numbers was first proved as a consequence of the following algebro-geometric result.

\begin{theorem}[ELSV formula~\cite{eke-lan-sha-vai}]
The simple Hurwitz numbers satisfy the following equation, where $m = 2g-2+n+|\mmu|$.
\[
H_{g,n}(\mu_1, \ldots, \mu_n) = m! \, \prod_{i=1}^n \frac{\mu_i^{\mu_i}}{\mu_i!} \sum_{|\mathbf{d}| + k = 3g-3+n} (-1)^k \left[ \int_{\overline{\mathcal M}_{g,n}} \psi_1^{d_1} \cdots \psi_n^{d_n} \lambda_k \right] \mu_1^{d_1} \cdots \mu_n^{d_n}
\]
\end{theorem}

Here, $\psi_1, \ldots, \psi_n \in H^2(\overline{\mathcal M}_{g,n}; \mathbb{Q})$ and $\lambda_k \in H^{2k}(\overline{\mathcal M}_{g,n}; \mathbb{Q})$ are the psi-classes and Hodge classes on the Deligne--Mumford moduli space of curves $\overline{\mathcal M}_{g,n}$. For more information on the geometry of $\overline{\mathcal M}_{g,n}$, see~\cite{har-mor}.

The initial proof of Ekedahl, Lando, Shapiro, and Vainshtein computed the degree of the Lyashko--Looijenga map. Subsequently, the ELSV formula was deduced using localisation on the moduli space of stable maps. This latter proof was generalised to give the following result for orbifold Hurwitz numbers.

\begin{theorem}[Orbifold ELSV formula~\cite{joh-pan-tse}]
The orbifold Hurwitz numbers satisfy the following equation, where $m = 2g-2+n+\frac{\mmu}{a}$.
\[
H_{g,n}^{[a]}(\mu_1, \ldots, \mu_n) = m! \, a^m \prod_{i=1}^n \frac{(\mu_i/a)^{\lfloor \mu_i/a \rfloor}}{\lfloor \mu_i/a \rfloor!} \sum_{|\mathbf{d}| + k = 3g-3+n} \frac{(-1)^k}{a^{|\mathbf{d}|}} \left[ \int_{\overline{\mathcal M}_{g,[-\mmu]}(\mathcal{B}\mathbb{Z}_a)} \overline{\psi}_1^{d_1} \cdots \overline{\psi}_n^{d_n} \lambda_k^U \right] \mu_1^{d_1} \cdots \mu_n^{d_n}
\]
\end{theorem}

The integral here is performed over the moduli space of stable maps to the classifying stack of $\mathbb{Z}_a$, or equivalently, the moduli space of admissible covers. It depends on the tuple $(\mu_1, \ldots, \mu_n)$ modulo $a$, so it follows that the sum on the right hand side is a symmetric quasi-polynomial modulo $a$ of degree $3g-3+n$.

A similar polynomial structure has also been proven for monotone Hurwitz numbers.

\begin{theorem}[Polynomiality for monotone Hurwitz numbers~\cite{gou-gua-nov13b}] \label{thm:mpoly}
For $(g,n) \neq (0,1)$ or $(0,2)$, there exists a symmetric polynomial $\vec{P}_{g,n}$ of degree $3g-3+n$ such that
\[
\H_{g,n}(\mu_1, \ldots, \mu_n) = \prod_{i=1}^n \binom{2\mu_i}{\mu_i} \vec{P}_{g,n}(\mu_1, \ldots, \mu_n).
\]
\end{theorem}

Goulden, Guay-Paquet and Novak asked for a geometric interpretation of monotone Hurwitz numbers analogous to an ELSV formula. We remark here that the relation between monotone Hurwitz numbers and the topological recursion~\cite{do-dye-mat} can be combined with a theorem of Eynard that relates the output of the topological recursion to intersection numbers on moduli spaces of curves~\cite{eyn11}. This produces the formula
\[ 
\H_{g,n} (\mu_1, \ldots, \mu_n)  =  \int_{\overline{\mathcal M}_{g,n}} \exp \Big( - \sum_m s_m \, \kappa_m \Big)\sum_{d_1,\ldots, d_n}  \prod_{i=1}^n \binom{2\mu_i}{\mu_i} \, \frac{(2d_i + 2\mu_i -1)!!}{(2\mu_i - 1)!!} \, \psi_i^{d_i}, 
\]
where $\kappa_m \in H^{2m}(\overline{\mathcal M}_{g,n}; \mathbb{Q})$ denote the Mumford--Morita--Miller classes and the rational numbers $s_m$ are defined via the expansion
\[
\exp \Big( \sum_{m=1}^n s_m \, \hbar^m \Big) = \sum_{m=0}^n (2m + 1)!! \, \hbar^m + O(\hbar^{n+1}). 
\]
This analogue of the ELSV formula for monotone Hurwitz numbers was independently obtained by Alexandrov, Lewanski and Shadrin~\cite{ale-lew-sha}. However, there is no known proof of this formula that does not rely on the topological recursion and the results of Eynard.

\subsection*{Topological recursion}

The topological recursion of Chekhov, Eynard and Orantin~\cite{che-eyn,eyn-ora07} takes as input a spectral curve and outputs multidifferentials $\omega_{g,n}$ for $g \geq 0$ and $n \geq 1$. For brevity and simplicity, we formulate here the topological recursion for the case of a rational spectral curve. The definition for higher genus spectral curves and generalisations to other types of spectral curves can be found elsewhere in the literature~\cite{eyn-ora07}.

\begin{itemize}
\item {\bf Input.} A \emph{rational spectral curve} consists of two meromorphic functions $x, y: \mathbb{CP}^1 \to \mathbb{CP}^1$ with the condition that the zeros of $\dd x$ are simple and distinct from the zeros of $\dd y$.
\item {\bf Base cases.} The base cases are defined by the equations
\[
\omega_{0,1}(z_1) = -y(z_1) \, \dd x(z_1) \qquad \text{and} \qquad \omega_{0,2}(z_1, z_2) = \frac{\dd z_1 \, \dd z_2}{(z_1-z_2)^2}.
\]
\item {\bf Recursion.} Recursively define the multidifferentials $\omega_{g,n}$ by the equation
\[
\omega_{g,n}(\zz_S) = \sum_{\alpha} \mathop{\text{Res}}_{z=\alpha} K(z_1, z) \Bigg[ \omega_{g-1,n+1}(z, \overline{z}, \zz_{S \setminus \{1\}}) + \mathop{\sum_{g_1+g_2=g}}_{I \sqcup J = S \setminus \{1\}}^\circ \omega_{g_1, |I|+1}(z, \zz_I) \, \omega_{g_2, |J|+1}(\z, \zz_J) \Bigg],
\]
where $S = \{1, 2, \ldots, n\}$ and $\mmu_I = \{\mu_{i_1}, \ldots, \mu_{i_k}\}$ for $I = \{i_1, \ldots, i_k\}$. The outer summation is over the zeros $\alpha$ of $\dd x$. The notation $\overline{z}$ refers to the local Galois conjugate of $z$ with respect to the function $x$. In other words, $\overline{z}$ is the non-identity meromorphic function defined locally at $\alpha$ by the equation $x(z) = x(\overline{z})$. The $\circ$ over the inner summation means that we exclude terms that involve $\omega_{0,1}$. Finally, the kernel $K(z_1, z)$ is defined by the equation
\[
K(z_1, z) = - \frac{\int_o^z \omega_{0,2}(z_1, \,\cdot\,)}{[y(z) - y(\overline{z})] \, \dd x(z)}.
\]
\end{itemize}

The topological recursion has found application to many problems from enumerative geometry and mathematical physics. In particular, Hurwitz numbers of many flavours are either known or conjectured to be governed by the topological recursion. Such results can have profound geometric consequences. For example, a proof of the conjectured relation between spin Hurwitz numbers and topological recursion would lead to a proof of the spin ELSV conjecture of Zvonkine~\cite{zvo,sha-spi-zvo}.

\begin{theorem}[Topological recursion for Hurwitz numbers~\cite{eyn-mul-saf, do-lei-nor, bou-her-liu-mul, do-dye-mat}] \label{thm:toprechur}
The following table shows the rational spectral curves that govern simple Hurwitz numbers, orbifold Hurwitz numbers, and monotone Hurwitz numbers.\footnote{In fact, the simple and orbifold Hurwitz numbers require a slightly modified version of the topological recursion, in which $x$ is referred to as a $\mathbb{C}^*$-coordinate. In such cases, rather than requiring $x(z)$ to be meromorphic --- or equivalently $\dd x(z)$ to be meromorphic --- one requires $\dd \log x(z)$ to be meromorphic. See the relevant papers for details. \label{foot:logarithmic}} In all cases, the expansions of the correlation differentials at $x_1 = \cdots = x_n = 0$ satisfy the equations on the right for $(g,n) \neq (0,2)$. Here, we have used the notation $x_i = x(z_i)$, for $i = 1, 2, \ldots, n$.
\begin{align*}
& \text{simple} & x(z) &= z \exp (-z) && y(z) = z & \omega_{g,n} &= \sum_{\mu_1, \ldots, \mu_n = 1}^\infty \frac{H_{g,n}(\mu_1, \ldots, \mu_n)}{(2g-2+n+|\mmu|)!} \prod_{i=1}^n \mu_i x_i^{\mu_i-1} \, \dd x_i \\
& \text{orbifold} & x(z) &= z \exp(-z^a) && y(z) = z^a & \omega_{g,n} &= \sum_{\mu_1, \ldots, \mu_n = 1}^\infty \frac{H_{g,n}^{[a]}(\mu_1, \ldots, \mu_n)}{(2g-2+n+\frac{|\mmu|}{a})!} \prod_{i=1}^n \mu_i x_i^{\mu_i-1} \, \dd x_i \\
& \text{monotone} & x(z) &= \frac{z-1}{z^2} && y(z) = -z & \omega_{g,n} &= \sum_{\mu_1, \ldots, \mu_n = 1}^\infty \H_{g,n}(\mu_1, \ldots, \mu_n) \prod_{i=1}^n \mu_i x_i^{\mu_i-1} \, \dd x_i
\end{align*}
\end{theorem}

\subsection*{Quantum curve}

Spectral curves appear in various guises across mathematics and physics. It is often the case that they can be quantised to produce a differential operator that annihilates an associated wave function. Gukov and Su{\l}kowski proposed that quantum curves can be calculated using the topological recursion formalism~\cite{guk-sul}. Inspired by arguments from physics, they assert that the correlation differentials arising from the topological recursion can be integrated to yield the free energies. For example, in the case of monotone Hurwitz numbers, one obtains the following.
\begin{equation} \label{eq:freeenergies}
F_{g,n}(x_1, \ldots, x_n) = \sum_{\mu_1, \ldots, \mu_n = 1}^\infty \H_{g,n}(\mu_1, \ldots, \mu_n) \prod_{i=1}^n x_i^{\mu_i}
\end{equation}
These can then be assembled to produce a wave function
\begin{equation} \label{eq:wavefunction}
Z(x, \h) = \exp \bigg[ \sum_{g=0}^\infty \sum_{n=1}^\infty \frac{\h^{2g-2+n}}{n!} \, F_{g,n}(x, \ldots, x) \bigg],
\end{equation}
which is annihilated by the quantum curve differential operator. Furthermore, the semi-classical limit of the quantum curve should reproduce the original spectral curve.

\begin{theorem}[Quantum curves for Hurwitz numbers~\cite{mul-sha-spi, do-dye-mat}]
The following are the quantum curves for simple Hurwitz numbers, orbifold Hurwitz numbers, and monotone Hurwitz numbers. In the simple and orbifold cases, we take $\xx = x$ and $\yy = -\h x \frac{\partial}{\partial x}$, while in the monotone case, we take $\xx = x$ and $\yy = -\h \frac{\partial}{\partial x}$.
\begin{align*}
& \text{simple} && \yy - \xx e^{\yy} \\
& \text{orbifold} && \yy - \exp(\tfrac{a-1}{2} \, \yy) \, \xx^a \, \exp(\tfrac{a+1}{2} \, \yy) \\
& \text{monotone} && \xx \yy^2 + \yy + 1
\end{align*}
In other words, the differential operators on the right annihilate the respective wave functions given by equation~\eqref{eq:wavefunction}.
\end{theorem}

Note that the difference in the definition of the $\yy$ operator for the various cases is due to the use of logarithmic coordinates mentioned in footnote~\ref{foot:logarithmic}. The semi-classical limits of these quantum curves are obtained by setting $\h = 0$ and replacing the operators $\xx$ and $\yy$ with commuting variables $x$ and $y$, respectively. One can check that this does indeed recover the spectral curves of Theorem~\ref{thm:toprechur}.

\section{Cut-and-join analysis} \label{sec:cutjoin}

This section is devoted to the derivation of the cut-and-join recursion for monotone orbifold Hurwitz numbers. One way to count tuples of transpositions subject to certain constraints is to count monodromy graphs with certain weights. The notion of monodromy graph appeared in the work of Cavalieri, Johnson, and Markwig~\cite{cav-joh-mar10,cav-joh-mar11}, where it was used to prove the chamber structure of the double Hurwitz numbers and to derive the corresponding wall-crossing formula. It was subsequently used by Guay-Paquet, Markwig, and Rau~\cite{gua-mar-rau,mar-rau} as a convenient tool for the calculation of real double Hurwitz numbers. We will define the notion of a monotone monodromy graph and apply it to derive a cut-and-join recursion for monotone orbifold Hurwitz numbers.

Note that monodromy graphs appear naturally in the framework of tropical geometry, where they represent tropical covers of the projective line --- that is, limits of families of maps between Riemann surfaces under certain degenerations of the complex structure. Thus, it would be natural to ask whether monotone monodromy graphs admit an interpretation as limits of families of maps between Riemann surfaces endowed with some additional geometric data. Although the cut-and-join recursion for monotone orbifold Hurwitz numbers may be derived without introducing the notion of monotone monodromy graphs, we use this approach to expose a potential connection to algebraic and tropical geometry.

As mentioned in Section~\ref{sec:hurwitz}, the number of monotone Hurwitz factorisations depends only on the conjugacy class of the permutation $\sigma_0$. For the remainder of this section, we fix a positive integer $a$ and the permutation
\[
\sigma_0 = (1, 2, \ldots, a) ~ (a+1, a+2, \ldots, 2a) ~ \cdots ~ (ak-a+1, ak-a+2, \ldots, ak).
\]
We consider monotone Hurwitz factorisations $(\sigma_0, \sigma_1, \ldots, \sigma_m)$ and to each, we assign a graph endowed with some additional information.

\begin{definition} \label{def:graph}
Let $g$ be a non-negative integer and let $\mmu = (\mu_1, \ldots, \mu_n)$ be a tuple of positive integers with $|\mmu| = ak$. A graph $\Gamma$ is a \emph{monotone monodromy graph of type $(g, \mmu)$} if the following conditions hold.

{\em Graph conditions.}
\begin{itemize}
\item The graph $\Gamma$ is a connected directed graph with first Betti number equal to $g$.
\item The graph $\Gamma$ has $k+n$ leaves and all remaining vertices (called \emph{inner vertices}) have degree 3.
\item The inner vertices are totally ordered compatibly with the partial ordering induced by the directions of the edges. (This order corresponds to that of the transpositions.)
\end{itemize}

{\em Weight conditions.}
\begin{itemize}
\item Each edge $e$ of $\Gamma$ is equipped with a positive integer weight $w(e)$. The weights of edges adjacent to leaves directed inwards (called \emph{in-ends}) are equal to $a$. The weights of edges adjacent to leaves directed outwards (called \emph{out-ends}) are the parts of $\mmu$.
\item At each inner vertex, the sum of the weights of incoming edges equals the sum of the weights of outgoing edges. This is known as the \emph{balancing condition}.
\end{itemize}

{\em Colouring conditions.}
\begin{itemize}
\item Each edge of $\Gamma$ has one of three colours --- normal, dashed or bold --- such that the colouring at every inner vertex is of one of the six types listed in Figure~\ref{fig:vertices}.

\begin{figure}[ht!]
\begin{center}
\begin{multicols}{3}
\begin{tikzpicture}
\draw [ultra thick] (0,0) -- (1,0) -- (2,-0.5);
\draw [thin] (1,0) -- (2,0.5);
\end{tikzpicture}

~

\begin{tikzpicture}
\draw [ultra thick] (0,0) -- (1,0);
\draw [thin] (2,0.5) -- (1,0);
\draw [very thick, dashed] (1,0) -- (2,-0.5);
\end{tikzpicture}

\begin{tikzpicture}
\draw [ultra thick] (0,0.5) -- (1,0);
\draw [thin] (0,-0.5) -- (1,0);
\draw [very thick, dashed] (1,0) -- (2,0);
\end{tikzpicture}

~

\begin{tikzpicture}
\draw [ultra thick] (0,0.5) -- (1,0) -- (2,0);
\draw [thin] (0,-0.5) -- (1,0);
\end{tikzpicture}

\begin{tikzpicture}
\draw [ultra thick] (0,0.5) -- (1,0) -- (2,0);
\draw [very thick, dashed] (0,-0.5) -- (1,0);
\end{tikzpicture}

~

\begin{tikzpicture}
\draw [ultra thick] (0,0.5) -- (1,0);
\draw [very thick, dashed] (0,-0.5) -- (1,0);
\draw [very thick, dashed] (1,0) -- (2,0);
\end{tikzpicture}
\end{multicols}
\end{center}
\caption{These are the possible types of inner vertices of a monotone monodromy graph. The edges are assumed to be oriented from left to right and the weights and counters of the edges are not specified.} \label{fig:vertices}
\end{figure}
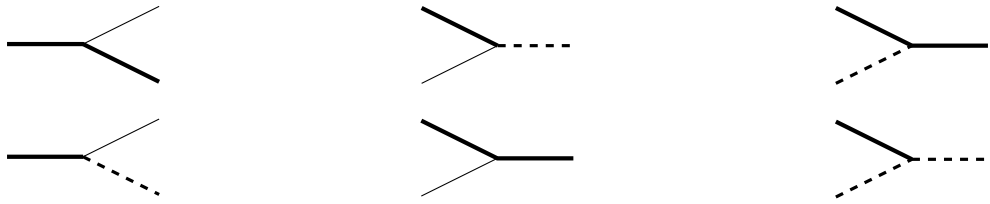

\item There are no normal in-ends and there is a unique bold out-end.

\item Chains of bold edges --- that is, inclusion maximal subgraphs containing only bold edges --- begin at in-ends of $\Gamma$. For any chain $C$ of bold edges, we can associate the numbers $f_C$ and $l_C$. These are respectively the numbers of the first and last inner vertices that belong to $C$, according to the ordering of the vertices of $\Gamma$. As a consequence of the monotonicity condition, the intervals $[f_C,l_C]$ are not allowed to intersect for different chains of bold edges. 
\end{itemize}

{\em Counter conditions.}
\begin{itemize}
\item Each dashed or bold edge is marked with a \emph{counter}, which is an integer from $1$ to $a$. The counter for each in-end is set to $1$. The counter for an ingoing bold edge at an inner vertex is less than or equal to the counter for the outgoing bold or dashed edge. Furthermore, if the weight of a bold or dashed edge is $w$, then its counter is greater then $a-w$. This condition arises from the fact that the cycle corresponding to the bold or dashed edge with counter $\ell$ should contain at least $a - \ell + 1$ elements.
\end{itemize}
\end{definition}

An example of a monotone monodromy graph is shown in Figure~\ref{fig:graph}.

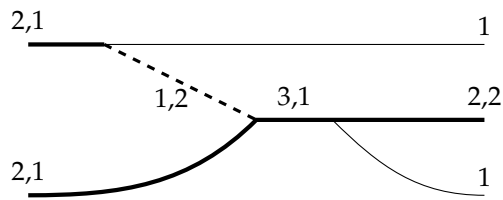
\begin{figure}[ht!]
\begin{center}
\begin{tikzpicture}
\draw [ultra thick] (-2,1) node[above]{2,1} -- (-1,1);
\draw [ultra thick] (-2,-1) node[above]{2,1} to [out=0,in=225] (1,0)-- (1.5,0) node[above]{3,1} -- (2,0) -- (4,0) node [above]{2,2};
\draw [thin] (-1,1) -- (2,1) -- (4,1) node[above]{1};
\draw [very thick, dashed] (-1,1) -- (-0.12,0.56) node[below]{1,2} -- (1,0);
\draw [thin] (2,0) to [out=-45,in=180] (4,-1) node[above]{1};
\end{tikzpicture}
\end{center}
\caption{This is a monotone monodromy graph of genus 0 that contributes to the calculation of $\H^{[2]}(1,1,2)$. Each normal edge is marked with its weight. Each dashed or bold edge is marked with a pair comprising its weight and counter. The edges are directed from left to right while the vertices are ordered from left to right.} \label{fig:graph}
\end{figure}

We now present the construction of the monotone monodromy graph associated to a monotone Hurwitz factorisation ($\sigma_0, \sigma_1, \ldots, \sigma_m)$. The graph is constructed inductively, by cutting or joining its out-ends by inner vertices. Draw $k$ in-ends of weight $a$ and establish a one-to-one correspondence between these edges and the cycles of $\sigma_0$. Let $\tau_i$ = $\sigma_0 \sigma_1 \cdots \sigma_i$ for $i = 0, 1, \ldots, m$. At each step, a one-to-one correspondence between the cycles of $\tau_i$ and the out-ends of the constructed graph is fixed. The transposition $\sigma_i$ for $i=1, 2, \ldots, m$ either cuts or joins two cycles of $\tau_{i-1}$. We encode this by the type of the $i$th inner vertex adjacent to the edge corresponding to the cycle under consideration. In the case of a cut, the vertex has one ingoing edge and two outgoing edges, while in the case of a join, it has two ingoing edges and one outgoing edge. The constructed trivalent vertex is a sink for the edges that were cut or joined, and a source for the ends emerging from it. So we have endowed all the edges with a direction. Moreover, the order of constructing inner vertices is precisely the order of the inner vertices in the definition of a monotone monodromy graph. The weights of edges of the constructed graph correspond to the lengths of the cycles of $\tau_i$. 

The colouring of the edges is established as follows. In the initial state, all the edges are coloured dashed and the counters of all in-ends are set to 1. Now every transposition $\sigma_i = (r_i ~ s_i)$ with $r_i < s_i$ either cuts a cycle into two, or joins two cycles.
\begin{itemize}
\item If the transposition $\sigma_i$ cuts a cycle into two, colour the ingoing edge to the vertex bold. Colour the outgoing edge corresponding to the newly formed cycle containing $r_i$ normal and colour the other outgoing edge dashed.
\item If the transposition $\sigma_i$ joins two cycles, change the colour of the ingoing edge corresponding to the cycle containing $s_i$ to bold and colour the outgoing edge dashed.
\end{itemize}
Finally, change the colour of the out-end corresponding to the cycle of $\tau_m$ containing $s_m$ to bold.

For the $i$th inner vertex, we set the counter of the outgoing bold or dashed edge as follows. The transposition $\sigma_i$ is of the form $(r_i, q_i a + \ell_i)$, where $q_i \in \{0, 1, \ldots, k-1\}$, $\ell_i \in \{1, 2, \ldots, a\}$, and $r_i < q_i a + \ell_i$. The corresponding counter is set to $\ell_i$.

It is clear that any monotone Hurwitz factorisation $(\sigma_0, \sigma_1, \ldots, \sigma_m)$ yields a monotone monodromy graph. On the other hand, any monotone monodromy graph corresponds to a number of monotone Hurwitz factorisations, and this number is given by Lemma~\ref{lem:multiplicity} below.

Define the multiplicity $m_v$ of an inner vertex $v$ of a monotone monodromy graph as follows. Set $m_v = 1$ if $v$ corresponds to a cut and set $m_v = w(e)$ otherwise, where $e$ is the ingoing normal or dashed edge at $v$. Also, we define the number $n_\Gamma$ for a monotone monodromy graph $\Gamma$ by the following rule. For any chain of bold edges $C_i$ in $\Gamma$, define the number $n_i$ to be the number of in-ends adjacent to the vertices in $C_i$. Suppose that $\Gamma$ contains $N_\Gamma$ chains of bold edges. The chains of bold edges are naturally ordered, as they induce a decomposition of the set of inner vertices into equivalence classes and this decomposition respects the ordering of the inner vertices. List the chains of bold edges in reverse order, with $C_1$ as the largest. Then define $n_\Gamma$ to be equal to
\[
n_\Gamma = \frac{k!}{C_{N_\Gamma} \, (C_{N_\Gamma} + C_{N_{\Gamma - 1}}) \, \cdots \, (C_{N_\Gamma} + C_{N_{\Gamma - 1}} + \cdots + C_1)}.
\]

\begin{lemma} \label{lem:multiplicity}
The number $m(\Gamma)$ of monotone Hurwitz factorisations $(\sigma_0,\sigma_1,\ldots,\sigma_m)$ that yield the monotone monodromy graph $\Gamma$ satisfies
\[
m(\Gamma) = n_\Gamma \prod_{v} m_v,
\]
where the product is over the inner vertices of $\Gamma$.
\end{lemma}
\begin{proof}

First, note that for any monotone Hurwitz factorisation $(\sigma_0, \sigma_1, \ldots, \sigma_m)$ with $\sigma_m = (r ~ s)$, where $r < s$ and $s \geq ak+1$, all the numbers in $\{c, c+1, \ldots, ak\}$ are in the same cycle of $\sigma_0\sigma_1\cdots\sigma_m$. Furthermore, for any $r \in \{c, c+1, \ldots,ak-1\}$, we have $\sigma_0\sigma_1\cdots\sigma_m(r) = r+1$.

So in the case that an edge of a monotone monodromy graph is cut, the corresponding transposition is uniquely defined by the weights of the outgoing edges from the corresponding vertex and the counter on the outgoing bold or dashed edge.

In that case that two edges of a monotone monodromy graph are joined, the largest element of the corresponding transposition is uniquely defined by the counter assigned to the outgoing edge. However, we have a number of possibilities for the first element of the transposition and this number is precisely the weight of the ingoing edge that is not bold.

Finally, we have to choose how to assign the cycles of $\sigma_0$ to the in-ends of the monotone monodromy graph. Obviously, the number of cycles of $\sigma_0$ should obey the following rules.
\begin{itemize}
\item For any chain $C$ of bold edges, the number of the cycle attached to the bold in-end is greater then the numbers of the cycles attached to the normal in-ends adjacent to vertices $f_C,\ldots,l_C$.
\item If the chain $C_i$ is larger than the chain $C_j$ according to the order on chains of edges, then the number of the cycle attached to the bold in-end of $C_i$ is larger than the number of the cycle attached to the bold in-end of $C_j$.

Summing these two conditions, we obtain that the number of ways to choose the correspondence between the cycles of $\sigma_0$ and the in-ends of $\Gamma$ is
\[
\frac{(k-1)!}{(k-n_{C_1})!} \times \frac{(k-1-n_{C_1})!}{(k-n_{C_1} - n_{C_2})!} \times \cdots \times \frac {(k - 1 - n_{C_1} - \cdots - n_{C_{N_\Gamma-1}})!}{(k - n_{C_1} - \cdots - n_{C_{N_\Gamma}})!},
\]
which is equal to $n_\Gamma$. \qedhere
\end{itemize}
\end{proof}

Note that the out-ends of a monotone monodromy graph $\Gamma$ admit a natural order, compatible with the ordering of the inner vertices to which they are adjacent. If two out-ends are adjacent to the same inner vertex, we declare the dashed or bold out-end to be greater then the normal out-end. Moreover, we are going to refine the enumeration of monotone Hurwitz numbers by keeping track of the counter of the unique bold out-end.

\begin{definition}
As above, let $\mmu = (\mu_1, \ldots, \mu_n)$ be a tuple of positive integers and let $S = \{1, 2, \ldots, n\}$. We denote by $N_g^\ell(\mu_1\vbar \mmu_{S\setminus \{1\}})$ the weighted number of monotone monodromy graphs of type $(g,\mmu)$ such that their unique bold out-end has weight $\mu_1$ and counter $\ell$. We also define
\[
N_g(\mmu) = \sum_{\ell = 1}^a \sum_{k=1}^n N_g^\ell (\mu_k\vbar \mmu_{S\setminus \{k\}}).
\] 
\end{definition}

Another way to describe the same refined enumeration without reference to the notion of monotone monodromy graph is the following. Let $\bar{\mmu} = (\mu_1\vbar \mmu_{S\setminus \{1\}})$ be a tuple of positive integers with a distinguished element $\mu_1$. Denote by $\sigma_0$ the special permutation
\[
\sigma_0 = (1, 2, \ldots, a) ~ (a+1, a+2, \ldots, 2a) ~ \cdots ~ (ak-a+1, ak-a+2, \ldots, ak).
\]
We define a \emph{refined Hurwitz factorisation of type $(g,\bar{\mmu},\ell)$} to be a monotone $a$-orbifold Hurwitz factorisation of type $(g,\mmu)$ --- see Definition~\ref{def:fact} --- such that
\begin{itemize}
\item the permutation $\sigma_0$ is fixed to be the special one introduced above;
\item the cycle of $\sigma_0 \sigma_1 \cdots \sigma_m$ labeled by 1 has length $\mu_1$; and
\item the transposition $\sigma_m$ is of the form $(r_m, ak - a + \ell)$ and all the numbers $ak-a + \ell, \ldots, ak$ are contained in the first cycle of $\sigma_0 \sigma_1 \cdots \sigma_m$.
\end{itemize}
Now define $N_g^\ell (\mu_1 \vbar \mmu_{S\setminus \{1\}})$ to be the number of refined Hurwitz factorisations of type $(g,\bar{\mmu},\ell)$. By the relation between monotone Hurwitz factorisations and monotone monodromy graphs, these numbers coincide with those defined via the notion of monotone monodromy graphs.

Now analyse the possibilities for the action of the transposition $\sigma_m$, where $m = 2g-2+n+\frac{|\mmu|}{a}$. On the level of monotone monodromy graphs, this corresponds to analysing the possibilities that arise on removal of the last inner vertex. One obtains the following cases in the calculation of the number $N^\ell_g(\mu_1,\mmu_{S\setminus \{1\}})$.

\begin{itemize}
\item {\em The permutation $\sigma_m$ is a cut.} \\ On the level of monotone monodromy graphs, this corresponds to the case when the last inner vertex is a cut. So the corresponding term in the cut-and-join recursion is
\[
\Theta(\mu_1 + \ell - a -1) \sum_{p=1}^\ell \sum_{i=2}^n N^p_g(\mu_1 + \mu_i \vbar \mmu_{S\setminus\{1,i\}}).
\]
Here, $\Theta$ denotes the Heaviside step function, which accounts for the allowed values of the counter.

\item {\em The permutation $\sigma_m$ is a redundant join.} \\ On the level of monotone monodromy graphs, this corresponds to the last inner vertex joining two edges that already belong to the same connected component of the graph. So the corresponding term in the cut-and-join recursion is
\[
\sum_{\alpha + \beta = \mu_1} \sum_{p=1}^\ell \beta \, N^p_{g-1}(\alpha \vbar \mmu_{S\setminus\{1\}}, \beta).
\]
Here, the factor $\beta$ appears due to the multiplicity of the vertex under consideration.

\item {\em The permutation $\sigma_m$ is an essential join.} \\ On the level of monotone monodromy graphs, this corresponds to the last inner vertex joining two connected components. The two components have degrees $k_1a$ and $k_2a$, where $k_1 + k_2 = k$. So the corresponding term in the cut-and-join recursion is
\[
\sum_{\alpha+ \beta = \mu_1}\mathop{\sum_{g_1 + g_2 = g}}_{I \sqcup J = S \setminus \{1\}}\sum_{p=1}^\ell \binom{k-1}{\frac{1}{a} (|\mmu_I| + \beta)} \beta \, N_{g_1}(\mmu_I, \beta) \, N_{g_2}^p(\alpha \vbar \mmu_J).
\]
To obtain this expression, we note that for any two monotone monodromy graphs, any decomposition of the set of cycles of $\sigma_0$ into two parts yields a unique order on the union of the sets of their chains of bold edges.
\end{itemize}

The cut-and-join analysis above leads directly to the following result.

\begin{proposition} \label{prop:ncutjoin}
The numbers $N^\ell_g(\mu_1 \vbar \mmu_{S\setminus \{1\}})$ satisfy the following recursion.
\begin{align*}
N^\ell_g(\mu_1 \vbar &\mmu_{S\setminus \{1\}}) =\, \Theta(\mu_1 + \ell - a -1) \sum_{p=1}^\ell \sum_{i=2}^n N^p_g(\mu_1 + \mu_i \vbar \mmu_{S\setminus\{1,i\}})\\
&+\sum_{\alpha + \beta = \mu_1} \sum_{p=1}^\ell \Bigg[ \beta \, N^p_{g-1}(\alpha \vbar \mmu_{S\setminus\{1\}}, \{\beta\}) + \mathop{\sum_{g_1 + g_2 = g}}_{I \sqcup J = S \setminus \{1\}} \binom{k-1}{\frac{1}{a}(|\mmu_I| + \beta)} \beta \, N_{g_1}(\mmu_I, \beta) \, N_{g_2}^p(\alpha \vbar \mmu_J) \Bigg]
\end{align*}
\end{proposition}

As mentioned in Section~\ref{sec:hurwitz}, the number of monotone Hurwitz factorisations depends only on the conjugacy class of the permutation $\sigma_0$. So in order to obtain the corresponding Hurwitz number, we need to multiply the number of monotone Hurwitz factorisations with fixed $\sigma_0$ by the number of permutations in the conjugacy class of $\sigma_0$ and divide by the order of the symmetric group. So it is natural to write
\begin{equation} \label{eq:hfromn}
\H_{g,n}^{[a],\ell}(\mu_1 \vbar \mmu_{S\setminus \{1\}}) = \frac 1{a^k k!} N^\ell_g(\mu_1 \vbar \mmu_{S\setminus \{1\}}) \qquad \qquad \text{and} \qquad \qquad \H_{g,n}^{[a]}(\mmu_S) =\frac 1{a^k k!} N_g(\mmu_{S}).
\end{equation}
It follows that the monotone orbifold Hurwitz number can be expressed in terms of the refined enumeration via equation~\eqref{eq:hurwitzsum1}, which states that
\begin{equation} \label{eq:hurwitzsum}
\H_{g,n}^{[a]}(\mu_1, \ldots, \mu_n) = \sum_{i=1}^n \sum_{\ell=1}^a \H_{g,n}^{[a],\ell}(\mu_i \vbar \mmu_{S \setminus \{i\}}).
\end{equation}

Theorem~\ref{thm:cutjoin} now follows immediately from Proposition~\ref{prop:ncutjoin}, equations~\eqref{eq:hfromn} and~\eqref{eq:hurwitzsum}, as well as the trivial base case calculation $\H_{0,1}^{[a],\ell}(a \vbar) = \frac{1}{a} \delta_{\ell, 1}$.

As a simple application of the cut-and-join recursion, we calculate the monotone orbifold Hurwitz numbers in the case $(g,n) = (0,1)$.

\begin{proposition} \label{prop:01problem}
For every positive integer $k$, we have $\H_{0,1}^{[a]}(ak) = \frac {1}{ak^2} \binom{ak + k -2}{k-1}$.
\end{proposition}

\begin{proof}
The cut-and-join recursion in the case $(g,n) = (0,1)$ reads
\[
ak \, \H^{[a],\ell}_{0,1}(ak \vbar) = \sum_{m+n = k} a^2mn \, \H^{[a]}_{0,1}(am) \sum_{p=1}^\ell \H^{[a],p}_{0,1}(an \vbar ).
\]
Now introduce the generating functions
\[
\phi_\ell = \sum_{k=1}^\infty ak \, \H^{[a],\ell}_{0,1}(ak \vbar ) \, t^k \qquad\qquad \text{and} \qquad\qquad \psi = \sum_{p=1}^a \phi_p = \sum_{k=1}^\infty ak \, \H^{[a]}_{0,1}(ak) \, t^k.
\]
Evidently, these functions obey the system of equations
\[
\phi_\ell = \psi \sum_{p=1}^\ell \phi_p + \delta_{1,\ell} t \qquad\qquad \text{ for $\ell = 1, 2, \ldots, a$}.
\]
A simple induction argument shows that $\phi_\ell = \psi^2 (1 - \psi)^{a-\ell} + \delta_{1,\ell}t$ for all $\ell = 1, 2, \ldots, a$. In particular, the equation for $\ell = 1$ yields $t = \psi (1 - \psi)^a$. Now invoke the Lagrange inversion theorem to deduce that
\[
\H^{[a]}_{0,1}(ak) = \frac 1{ak^2}\binom{ak + k -2}{k-1}. \qedhere
\]
\end{proof}

\section{The quantum curve} \label{sec:qcurve}

As discussed in Section~\ref{sec:hurwitz}, we define the correlation functions for the monotone orbifold Hurwitz numbers thus.
\[
\vec{F}_{g,n}^{[a]}(x_1, \ldots, x_n) = \sum_{\mu_1, \ldots \mu_n = 1}^\infty \H_{g,n}^{[a]}(\mu_1, \ldots, \mu_n) \, x_1^{\mu_1} \cdots x_n^{\mu_n}
\]
From these, we define the wave function in the following way.
\begin{align*}
\vec{Z}^{[a]}(x, \h) &= \exp \left[ \sum_{g=0}^\infty \sum_{n=1}^\infty \frac{\h^{2g-2+n}}{n!} \, \vec{F}^{[a]}_{g,n}(x, \ldots, x) \right] \\
&= \exp \left[ \sum_{g=0}^\infty \sum_{n=1}^\infty \frac{\h^{2g-2+n}}{n!} \sum_{\mu_1, \ldots, \mu_n = 1}^\infty \H_{g,n}^{[a]}(\mu_1, \ldots, \mu_n) \, x^{|\mmu|} \right]
\end{align*}
Lemma~\ref{lem:series} below allows us to interpret this as an element of $\mathbb{Q}(\!(\h)\!)[[x]]$ --- in other words, a formal power series in $x$ whose coefficients are Laurent series in $\h$.

A combinatorial interpretation for the coefficients of the wave function can be obtained by making the following observations.
\begin{itemize}
\item Setting the arguments of $\vec{F}^{[a]}_{g,n}$ equal to $x$ enumerates branched covers only by their degree, rather than their ramification profile over $\infty$. Furthermore, the factor $\frac{1}{n!}$ removes the labelling over $\infty$.
\item The factor $\h^{2g-2+n}$ collects terms only by their Euler characteristic.
\item The exponential passes from an enumeration of connected objects to an enumeration of possibly disconnected objects.
\end{itemize}
In short, the wave function is a generating function for possibly disconnected branched covers, assembled according to the degree and the Euler characteristic $2g-2+n$. Applying this strategy leads to the following expression for the wave function.

\begin{lemma} \label{lem:series}
The wave function for monotone orbifold Hurwitz numbers is given by
\[
\vec{Z}^{[a]}(x, \h) = 1 + \sum_{k=1}^\infty \frac{x^{ak}}{k! \, a^k \, \h^k} \prod_{j=1}^{ak-1} \frac{1}{1-j\h}.
\]
\end{lemma}

\begin{proof}
The first step is to use the exponential formula to interpret the wave function as a generating function for disconnected monotone orbifold Hurwitz numbers. The number $\H_{g,n}^{[a]\bullet}(\mu_1, \ldots, \mu_n)$ is defined similarly to its connected counterpart, though without the transitivity condition.
\[
\vec{Z}^{[a]}(x, \h) = 1 + \sum_{g=-\infty}^\infty \sum_{n=1}^\infty \frac{\h^{2g-2+n}}{n!} \sum_{\mu_1, \ldots, \mu_n = 1}^\infty \H_{g,n}^{[a]\bullet}(\mu_1, \ldots, \mu_n) \, x^{|\mmu|}
\]
Note that we are using the arithmetic genus, which may be negative for a disconnected branched cover. The use of the exponential formula to pass from a connected to a disconnected enumeration is common in the literature --- for example, see \cite{do-dye-mat} for an explicit proof in the case of monotone Hurwitz numbers.

Now observe that the coefficient of the monomial $x^d \h^r$ in the expansion of the wave function is precisely $\frac{1}{d!}$ multiplied by the number of tuples $(\sigma_0, \sigma_1, \ldots, \sigma_m)$ of permutations in $S_d$ such that
\begin{itemize}
\item $m = r + \frac{d}{a}$;
\item $\sigma_0$ has cycle type $(a, a, \ldots, a)$; and
\item $\sigma_1, \ldots, \sigma_m$ is a monotone sequence of transpositions.
\end{itemize}
It was shown in \cite{do-dye-mat} that the number of monotone sequences of $m$ transpositions in $S_d$ is equal to the Stirling number of the second kind $\stirling{d+m-1}{d-1}$ for $d \geq 1$ and $m \geq 0$. Therefore, one may express the wave function as
\[
\vec{Z}^{[a]}(x, \h) = 1 + \sum_{k=1}^\infty \sum_{m=0}^\infty \stirling{ak+m-1}{ak-1} \, \frac{x^{ak } \, \h^m}{k! \, a^k \, \h^k}.
\]

To obtain the desired expression for the wave function, invoke the following well-known generating function for Stirling numbers of the second kind.
\[
\sum_{N=0}^\infty \stirling{N}{K} \, \h^{N-K} = \prod_{j=1}^K \frac{1}{1-j\h} \qedhere
\]
\end{proof}

As an immediate corollary of this lemma, we obtain the fact that $\vec{Z}^{[a]}(x, \h) \in \mathbb{Q}(\!(\h)\!)[[x]]$.

Lemma~\ref{lem:series} constitutes the first part of Theorem~\ref{thm:qcurve}. We now use it to prove the remainder of Theorem~\ref{thm:qcurve}, which states that
\[
\left[ \xx^{a-1} + \prod_{j=0}^{a-1} (1 + \xx \yy + j\h) \, \yy \right] \vec{Z}^{[a]}(x, \h) = 0.
\]

\begin{proof}[Proof of Theorem~\ref{thm:qcurve}]
First, we consider the action of $\xx^{a-1}$ on the wave function.
\begin{align} \label{eq:qcurve1}
\xx^{a-1} \vec{Z}^{[a]}(x, \h) &= x^{a-1} + \sum_{k=1}^\infty \frac{x^{ak+a-1}}{k! \, a^k \, \h^k} \prod_{j=1}^{ak-1} \frac{1}{1-j\h} \nonumber \\
&= \sum_{k=1}^\infty \frac{x^{ak-1}}{(k-1)! \, a^{k-1} \, \h^{k-1}} \prod_{j=1}^{ak-a-1} \frac{1}{1-j\h} 
\end{align}

Clearly, the operators $(1+ \xx \yy + j\h)$ commute for different $j$, and satisfy
\[
(1+\xx \yy + j\h) \, x^{ak-1} = (1 - (ak-1-j)\h) \, x^{ak-1}.
\]

Therefore, the action of $\displaystyle\prod_{j=0}^{a-1} (1 + \xx \yy + j\h) \, \yy$ on the wave function is as follows.
\begin{align} \label{eq:qcurve2}
\prod_{j=0}^{a-1} (1 + \xx \yy + j \h) \, \yy \, \vec{Z}^{[a]}(x, \h) &= - \h\prod_{j=0}^{a-1} (1 + \xx \yy + j \h) \sum_{k=1}^\infty \frac{ak \, x^{ak-1}}{k! \, a^k \, \h^k} \prod_{j=1}^{ak-1} \frac{1}{1-j\h} \nonumber \\
&= - \prod_{j=0}^{a-1} (1 + \xx \yy + j \h) \sum_{k=1}^\infty \frac{x^{ak-1}}{(k-1)! \, a^{k-1} \, \h^{k-1}} \prod_{j=1}^{ak-1} \frac{1}{1-j\h} \nonumber \\
&= - \prod_{j=0}^{a-1} (1 - (ak-1-j)\h) \sum_{k=1}^\infty \frac{x^{ak-1}}{(k-1)! \, a^{k-1} \, \h^{k-1}} \prod_{j=1}^{ak-1} \frac{1}{1-j\h} \nonumber \\
&= - \sum_{k=1}^\infty \frac{x^{ak-1}}{(k-1)! \, a^{k-1} \, \h^{k-1}} \prod_{j=1}^{ak-a-1} \frac{1}{1-j\h}
\end{align}

Adding equations~\eqref{eq:qcurve1} and \eqref{eq:qcurve2} yields the desired result.
\end{proof}

\begin{remark}
One can attempt to interpret the wave function either as an expansion in $\h$ or as an expansion in $x$. These naturally lead to two distinct approaches to determining the quantum curve. The difficulty of the former approach is that the expansion is not well-defined if one includes the unstable terms $(g,n) = (0,1)$ and $(0,2)$. This issue can be overcome by interpreting the wave function as the product of unstable and stable parts, and incorporating the unstable part into the quantum curve differential operator. In this paper, we follow the latter approach, which gives a cleaner combinatorial argument.
\end{remark}

\section{A conjecture on topological recursion} \label{sec:toprec}

Conjecture~\ref{con:toprec} of Section~\ref{sec:intro} states that topological recursion applied to the spectral curve
\[
x(z) = z(1-z^a) \qquad \qquad \text{and} \qquad \qquad y(z) = \frac{z^{a-1}}{z^a-1}.
\]
produces correlation differentials that satisfy
\[
\omega_{g,n} = \sum_{\mu_1, \ldots, \mu_n = 1}^\infty \H_{g,n}^{[a]}(\mu_1, \ldots, \mu_n) \prod_{i=1}^n \mu_i x_i^{\mu_i-1} \, \dd x_i, \qquad \text{for } (g,n) \neq (0,2).
\]
The formulation of the topological recursion that we refer to was defined in Section~\ref{sec:hurwitz}. In this section, we provide evidence to support our conjecture as well as some consequences.

Strong evidence for Conjecture~\ref{con:toprec} comes from Theorem~\ref{thm:qcurve}, which states that the quantum curve equation for monotone orbifold Hurwitz number is given by
\[
\left[ \xx^{a-1} + \prod_{j=0}^{a-1} (1 + \xx \yy + j\h) \, \yy \right] \vec{Z}^{[a]}(x, \h) = 0.
\]
In general, it is expected that the semi-classical limit of a quantum curve for a given enumerative problem should recover the spectral curve. Furthermore, topological recursion applied to this spectral curve should produce (derivatives of) the free energies for the enumerative problem. This viewpoint is discussed in the work of Gukov and Su{\l}kowski in the context of quantisations of A-polynomials for knots~\cite{guk-sul}. The semi-classical limit of the quantum curve for monotone orbifold Hurwitz numbers is $x^{a-1} + y(1+xy)^a = 0$, and it is easy to check that this has the rational parametrisation given by equation~\eqref{eq:scurve}.

The spectral curve should come from the $(g,n) = (0,1)$ information of the enumerative problem, as discussed in~\cite{dum-mul-saf-sor}. Therefore, further evidence for Conjecture~\ref{con:toprec} is provided by the following result.

\begin{proposition}
If we write $y = -\frac{\partial}{\partial x} F_{0,1}^{[a]}(x)$, then
\[
x^{a-1} + y(1+xy)^a = 0,
\]
thereby recovering the spectral curve of equation~\eqref{eq:scurve}.
\end{proposition}

\begin{proof}
In the proof of Proposition~\ref{prop:01problem}, we obtained the equation $t = \psi(1-\psi)^a$, where
\[
\psi(t) = \sum_{k=1}^\infty ak \, \H_{0,1}^{[a]}(ak) t^k \qquad \Rightarrow \qquad \psi(x^a) = x \frac{\partial}{\partial x} \vec{F}_{0,1}^{[a]}(x).
\]
It follows that
\[
x^a = \Big[x\frac{\partial}{\partial x} \vec{F}^{[a]}_{0,1}(x)\Big] \Big[1 - x\frac{\partial}{\partial x} \vec{F}^{[a]}_{0,1}(x)\Big]^{a},
\]
which gives us the desired result.
\end{proof}

Of course, one can also obtain numerical evidence to support Conjecture~\ref{con:toprec}. We have implemented the cut-and-join recursion on the computer to calculate monotone orbifold Hurwitz numbers. We have also computed the correlation differentials $\omega_{g,n}$ for the spectral curve of equation~\eqref{eq:scurve} for small values of $g$, $n$ and $a$. All evidence generated in this way has been consistent with Conjecture~\ref{con:toprec}.

A consequence of Conjecture~\ref{con:toprec} would be the following structure result for monotone orbifold Hurwitz numbers, analogous to the polynomiality observed for simple Hurwitz numbers. In the case $a = 1$, it is equivalent to the known polynomiality for monotone Hurwitz numbers stated in Theorem~\ref{thm:mpoly}.

\begin{conjecture} \label{con:mopoly}
For positive integers $a$ and $\mu$, define
\[
C^{[a]}(\mu) = (a+1)^{\{ \mu / a \}} \binom{\mu + \lfloor \mu / a \rfloor}{\lfloor \mu / a \rfloor}.
\]
The monotone orbifold Hurwitz numbers satisfy
\[
\H_{g,n}^{[a]}(\mu_1, \ldots, \mu_n) = \prod_{i=1}^n C^{[a]}(\mu_i) \, Q_{g,n}^{[a]}(\mu_1, \ldots, \mu_n),
\]
where $Q_{g,n}^{[a]}$ is a symmetric quasi-polynomial modulo $a$ of degree $3g-3+n$.
\end{conjecture}

\begin{small}
\bibliographystyle{habbrv}
\bibliography{monotone-orbifold-hurwitz-v2}

\textsc{School of Mathematical Sciences, Monash University, VIC 3800, Australia} \\
\emph{Email:} \href{mailto:norm.do@monash.edu}{norm.do@monash.edu}

\textsc{St. Petersburg Department of the Steklov Mathematical Institute, Fontanka 27, St. Petersburg 191023, Russia} \\
\emph{Email:} \href{mailto:max.karev@gmail.com}{max.karev@gmail.com}

\end{small}

\end{document}